\documentclass[10pt, reqno]{amsart}
 \usepackage{tikz}
 \usepackage{hyperref, mathtools}
 \usepackage[skip=4pt]{caption}
 \usepackage{amssymb,amsmath,graphicx,amsthm}
  \usetikzlibrary{positioning}
  \usepackage{tkz-euclide}
\usepackage{rotating}
\usepackage{color,xcolor}
\usepackage{geometry}
\definecolor{darkred}{rgb}{1,0,0} 
\definecolor{darkgreen}{rgb}{0,0.8,0}
\definecolor{darkblue}{rgb}{0,0,1}

\hypersetup{colorlinks,
linkcolor=darkblue,
filecolor=darkgreen,
urlcolor=darkred,
citecolor=darkgreen}

 \def\M{\mathcal M}

 \def\bt{\begin{theorem}}
 	\def\el{\end{lemma}}
 \def\bl{\begin{lemma}}
 	\def\et{\end{theorem}}
 \def\bp{\begin{proposition}}
 	\def\ep{\end{proposition}}
 \def\bd{\begin{definition}}
 	\def\ed{\end{definition}}
 \def\br{\begin{remark}}
 	\def\er{\end{remark}}

 \newcommand{\e}[1]{\bold{e_{#1}}}

 \def\R{{\mathbb R}}
\def\M{{\mathcal{M}}}

 \def\B{\mathbb B}
 
 \def\N{{\mathbb N}}
 \def\P{{\mathcal P}}
 \def\O{{\mathbb O}}
\def\S{{\mathbb S}}
\def\S1{{\mathbb S^1}}
\def\H{{\mathbb H}} 
\def\i{{\bold i}}
\def\j{{\bold j}}
\def\k{{\bold k}}
\def\e{{\bold e}}

 \def\di{\partial}
\def\ux{\underline{x}}
 \def\dib{\bar\partial}
 \def\label#1{\label{#1}{\bf(#1)}~}

 \numberwithin{equation}{section}

 \theoremstyle{plain}
 \newtheorem*{theorem*}{Theorem}

 \newtheorem{theorem}{Theorem}[section]
 
 \newtheorem{corollary}[theorem]{Corollary}
 
 \newtheorem{lemma}[theorem]{Lemma}
 
 \newtheorem{proposition}[theorem]{Proposition}

 \theoremstyle{definition}
 
 \newtheorem{definition}[theorem]{Definition}
 
 \theoremstyle{remark}

 \newtheorem{remark}[theorem]{Remark}

 \newcommand{\p}{\partial}

 \newcommand{\dbar}{\bar\partial}

 \theoremstyle{plain} 

\begin{document}
 	
 	\title{Higher order gradients of monogenic functions }  
\author{Luca Baracco}
\address{Dipartimento di Matematica Tullio Levi-Civita, Universit\`a di Padova, via Trieste 63, 35121 Padova, Italy}
\email{baracco@math.unipd.it}

\author{Stefano Pinton}
\address{Dipartimento di Matematica, Politecnico di Milano, via E. Bonardi 9, 20133 Milano,, Italy}
\email{stefano.pinton@polimi.it}

 \begin{abstract} Given a monogenic function on the quaternionic algebra $\H$, the Clifford algebra $\R_n$ or the octonionic algebra $\O$ we prove that $|\nabla^m f|^\alpha$ is subharmonic for some $\alpha>0$ where $\nabla^m f$ is the $m$-th order gradient of $f$. We find also the optimal value of $\alpha$. This is generalization of a result of Calderon and Zygmund. 
	
\end{abstract}
  	\maketitle
  	\maketitle

 \tableofcontents 	
 	\section{Introduction and main results}

In this paper we are interested in the study of some properties of Fueter regular functions of a quaternionic or octonionic variable and monogenic functions in Clifford algebras. Fueter's functions and monogenic functions are important examples of the attempt to generalize holomorphic functions to the more general setting of non commutative algebras. These types of functions are defined as the solutions of two partial differential equations involving the so called Weyl's ($\bar\partial$) and Dirac's ($D$) operators. The difference between the two is that in $\dib$ also the derivative with respect to the real part is taken. Apart from that they are similar and their solutions have properties which are analogous to those of harmonic functions. In fact since these operators  factorize the Laplacian it turns out that regular and monogenic functions have harmonic components. Often we will refer to these two classes of functions as monogenic. We are interested in the subharmonicity of the norm of the generalized gradient of monogenic functions. Subharmonicity is a useful property and is important for instance when looking for estimates. A classical example of this kind of application is in the proof of Hartogs theorem on separate holomorphic functions where subharmonicity of $|\partial^n f|^\alpha$ is used for every $\alpha$ and all order of derivative $n$ (see \cite{BFP20} and the references therein). Such property is true only for harmonic functions in two variables and for holomorphic functions. It is easy to see that this is no longer true already for harmonic functions in more than two variables. Nonetheless it was proved in \cite{SW60} that for a harmonic function $f$ in $\R^n$
the power of the gradient $|\nabla f|^\alpha$ is subharmonic for $\alpha \ge \frac {n-2}{n-1}$.
Later \cite{CZ64} extended this result to higher order gradients proving that $|\nabla^m f|^\alpha$  is subharmonic for $\alpha\ge \frac{n-2}{n+m-2}$ and that such lower bound is optimal. For regular functions in \cite{SW68} it is proved that $|f|^\alpha$ is subharmonic for $\alpha\ge \frac 23$ and in \cite{KT06} the same result holds in the octonions for $\alpha\ge \frac 67$. In this note we want to extend the technique of \cite{CZ64} to higher order gradients of monogenic functions on quaternions, Clifford algebras and octonions. 
Our results are the following:
\begin{theorem}\label{t1} Let $\Omega\subset \H$ be an open set and $f:\Omega \rightarrow \H$ be a monogenic function. Then for every positive integer $m$ we have 
	\begin{equation} 
	|\nabla^m f |^\alpha \text{ is subharmonic for $\alpha \ge \frac{2}{m+3}$}  
	\end{equation}
\end{theorem}
\begin{theorem}\label{t3} Let $\Omega\subset \R^n$ and let $f:\Omega\rightarrow \R_n$ be a monogenic function.
	For all $m\in \N$ we have 
	\begin{equation}
	|\nabla^m f(x)|^\alpha \text{ is subharmonic}
	\end{equation}
	for $\alpha \ge \frac{n-2}{n+m-1}$
\end{theorem}
\begin{theorem}\label{t4}
	Let $\Omega\subset \O$ and let $f:\Omega\rightarrow \O$ be a monogenic function. For all $m\in\N$ we have
	$$ | \nabla^m f(x)|^\alpha \text{ is subharmonic}$$
	for $\alpha\ge \frac{6}{7+m}$.
\end{theorem}
In Theorems \ref{t1} and \ref{t3} we shall use the orthogonal basis of regular/monogenic polynomials introduced by Sommen (see \cite{DSS92}) while on Theorem \ref{t4} we adapt this technique to the non associative algebra $\O$.
     
\section{Quaternions and Fueter regular functions}\label{s2}
We begin by recalling some notations.  
Let $\H$ be the algebra of quaternions. Naturally $\H$ is isomorphic to $\R^4$ and a quaternion $x$ can either be expressed in terms of coordinates $(x_0,x_1,x_2,x_3)$ or as a sum of numbers multiplied by some imaginary units like $x_0+x_1\i +x_2\j +x_3\k$. In this last case the computation rules are as follows: $\i^2=\j^2=\k^2 =-1$, $\i\cdot\j=-\j\cdot\i =\k,\ \j\cdot\k=-\k\cdot\j=\i,\ \k\cdot \i =-\i\cdot\k=\j$. For this reason $x_0$ is called the real part and is denoted by $(x)_0$ while $x_1\i +x_2\j +x_3\k$ is called the imaginary part. Similarly for the other components $(x)_i=x_i$.

 We define the usual conjugation $\overline{\cdot} :\H \mapsto \H$ as
$$\overline{x}=\overline{x_0+x_1 \i+x_2\j+x_3\k}=x_0- x_1 \i-x_2\j-x_3\k $$
and employ it to define a real scalar product and a norm on $\H$ in the following way
$$(x,y)= (\overline{y}x)_0,\ |x|=(x,x)^{\frac12} .$$ 
We introduce the Cauchy-Fueter operators:
$$ \dbar := \p_{x_0} + \i \p_{x_1}+\j\p_{x_2}+\k\p_{x_3} \quad \p=\p_{x_0} - \i \p_{x_1}-\j\p_{x_2}-\k\p_{x_3} $$
and the Dirac operator
$$D:=  \i \p_{x_1}+\j\p_{x_2}+\k\p_{x_3}.$$
Let $\Omega$ be an open subset of $\H$ and let $f:\Omega \rightarrow \H$ be a  $C^1$ function.
\bd
We shall say that $f$ is left regular if
\begin{equation}
\dbar f= \p_{x_0}f + \i \cdot \p_{x_1}f +\j\cdot \p_{x_2}f+\k\cdot\p_{x_3}f =0
\end{equation}
\ed
Most of the properties of regular functions come from the fact that $\dbar$ appears as a factor in the factorization of the Laplace operator,
\begin{equation}
\Delta (f) =\dbar \p f=\p\dbar f
\end{equation}
in particular regular functions have harmonic components. For a natural number $m$ let $\beta=(\beta_1,...,\beta_m) \in \{0,1,2,3\}^m$ be a multi index and let
\begin{equation}
\partial^\beta f =\partial_{x_{\beta_1}}... \partial_{x_{\beta_m}}f \ ,
\end{equation}
$m$ is said to be the length of $\beta$ and is denoted by $|\beta|$.
If $f$ is regular it is easy to see that $\p^\beta f$ is regular too. We shall indicate with $\nabla^m f$ the set of all $m$ derivatives of $f$ and set
\begin{equation} 
| \nabla^m f|^2 :=\sum_{|\beta|=m}|\p^\beta f|^2.
\end{equation}
In order to prove Theorems \ref{t1} we follow the proof of \cite{CZ64} with some modifications. The key tool is the following lemma and for the sake of completeness we give a short proof (see \cite{CZ64} pages $212$ and $213$).
\begin{lemma} \label{l1}Let $\phi :\R_{\ge 0}\rightarrow \R$ be a $C^2$ concave, increasing  function and $U_j:\mathcal \R^n\to \R$ be harmonic functions for $j=1,\dots,\, l$ (here $l$ is an arbitrary positive integer). Let 
$$u:= \sum_{j=1}^l|U_j|^2\quad\textrm{and $U$ the vector valued function $U=(U_1,\dots, U_l)^T$}.$$  
We have $M:= \sup_{\Omega'} \left\{ \frac{|\nabla u |^2 }{2u\Delta (u)} \right\}\leq 1$ where $\Omega':=\{x\in\R^n: U(x)\neq 0,\  \nabla(U)\neq 0\}$ and moreover if 
$$ 2M t\phi'' (t)+\phi '(t)\ge 0$$
then $\phi(u)$ is subharmonic.
\end{lemma} 
\begin{proof} 
We begin by observing that  
$$\Delta(\phi(u))=\phi''(u)|\nabla u|^2+\phi'(u)\Delta (u) .$$ 
Since 
$$\Delta (u)=2\sum_{j=1}^l\left|\nabla(U_j)\right|^2+2\sum_{j=1}^l U_j\Delta(U_j)=2\sum_{j=1}^l\left|\nabla(U_j)\right|^2=2\sum_{i=1}^n|\partial_{x_i}U|^2$$
and 
$$ |\nabla u|^2=4\sum_{i=1}^n\left(\sum_{j=1}^l U_j\partial_{x_i}(U_j)\right)^2= 4\sum_{i=1}^n\left(U, \partial_{x_i}U\right)^2$$
we see that whenever $U(x)=0$ then $\Delta(\phi(u(x)))=\phi'(u(x))\Delta (u(x))\geq 0$ and whenever $\nabla (U(x))=0$ then $\Delta(\phi(u(x)))=0$ (here we denoted by $|\,\cdot\,|^2$ and $(\cdot,\cdot)$ the modulus and the standard scalar product in $\R^l$). If $U(x)\neq 0$ and $\Delta(\phi(u(x)))\neq 0$ then
$$ \frac{|\nabla u |^2 }{2u\Delta (u)}=\frac{4\sum_{i=1}^n\left(U, \partial_{x_i}U\right)^2}{4 |U|^2 \sum_{i=1}^n|\partial_{x_i}U|^2} \overset{\textrm{Cauchy-Schwartz}}{\leq}1 $$  
this implies that $M\leq 1$. Moreover 
$$ \Delta(\phi(u))=\Delta(u)\left(2\phi''(u)u \frac{|\nabla u |^2 }{2u\Delta (u)}+\phi'(u)\right)\geq \Delta(u)\left(2Mt\phi''(u)+\phi'(u)\right)\geq 0$$
which implies that $\phi(u)$ is subharmonic.
\end{proof}
\begin{remark}
The previous lemma holds when 
$$2M t\phi'' (t)+\phi '(t)= 0$$ 
in particular for $\phi(t)=Ct^{1-\frac 1{2M}}$.  
\end{remark}
We will apply the previous lemma for $u=|\nabla^m f |^2 $ and $\phi (t)= t^{1-\frac 1{2M}}$ (note that in our case the index $j$ of the Lemma \ref{l1} will take in account the multi index $\beta$ of the derivatives of order $m$ of $f$ and the fact that $f$ takes values in $\H$) . Our Theorem \ref{t1} follows if we prove that $M\leq\frac{m+3}{2(m+2)}$. This is our goal. 
\begin{proposition}\label{p1}
Let $f:\Omega\to \H$, where $\Omega$ is an open domain of $\H$, be a Fueter-regular function and let $m$ be a natural number. Let $u$  be 
\begin{equation}\label{1} u=|\nabla^m f|^2 =\sum_{|\beta| =m} (\p^\beta f,\p^\beta f) 
\end{equation}
then, when defined, we have 
\begin{equation} \label{eq1} \frac{|\nabla u|^2}{2u\Delta u} \le \frac{(m+3)}{2(m+2)}
\end{equation}
\end{proposition}
The proof needs the following lemmas from \cite{CZ64} adapted to this particular case and it is postponed at the end of this section. If $r>0$ by $\B_r^n$ we mean the ball centered at $0$ of radius $r$ in a vector space of dimension $n$. When $r=1$ we will omit it sometimes and when no confusion arise, we will also omit the dimension of the ball at the exponent. We will use the same notation for the sphere denoted by $\mathbb S^{n-1}_r$ 
\begin{lemma}\label{lemma1} Let $U:\H\rightarrow \H$ be a harmonic quaternionic valued function which means that $U=U_0+\i U_1 +\j U_2+\k U_3$ where each $U_l$ is harmonic. Assume that $U$ is homogeneous of degree $m$. Then:
\begin{itemize}
\item[1] there is constant $C_m$ depending only on $m$ such that
$$ \sum_{|\beta| =m} |\p^\beta U(0)|^2=C_{m} \int_{\B_1}|U|^2 d\lambda_4 .$$
\item[2]  if $V$ is another harmonic homogeneous quaternionic valued function then
$$ \sum_{|\beta |=m} (\p^\beta U(0),\p^\beta V(0))= C_m \int_{\B_1} (U,V) d\lambda_4 $$
\end{itemize}
\end{lemma}    
\begin{proof} By Lemma 1 in \cite{CZ64} we have the first identity for all components $U_i$ for $i=0,...,3$. Taking the sum we have the conclusion. Similarly for the second identity
\end{proof}
 Before giving the proof of Proposition \ref{p1} we need some preliminary considerations.
By the definition of $u$ we have that the terms in the left side of \ref{eq1} are given by
$$ |\nabla u|^2 =\sum_{i=0}^3 (\p_i u)^2=\sum_{i=0}^3 (\sum_{|\beta|=m} 2( \p^\beta \p_i f,\p^\beta f))^2 $$
and
\begin{equation}\label{2}
\begin{split}
\Delta u =\sum_{i=0}^3 \p^2_i \sum_{|\beta |=m} (\p^\beta f,\p^\beta f)&=2\sum_{|\beta|=m,i=0,..3}(\p^\beta \p_i f,\p^\beta \p_i f) 
\end{split}
\end{equation}
where the first identity follows from the fact that each $\p^\beta f$ is regular and hence harmonic. Now using Lemma \ref{lemma1} we want to express the terms in \ref{1} and \ref{2} as integrals over the unit ball. It is not restrictive, after a translation, to assume that the point under consideration is $0$ and we start by considering \ref{1} and \ref{2} at $0$. Since only the derivatives of order $m$ and $m+1$ of $f$ are needed we consider the Taylor series at $0$ of $f$ which is of the form
$$f (x)=\sum^\infty_{m=0} f_m(x)$$
where $f_m$ are homogeneous Fueter-regular polynomials of degree $m$.  
We see that \ref{1} and \ref{2} become respectively
\begin{equation}\label{3}\begin{split}
|\nabla u (0)|^2&=\sum^3_{i=0}\left(\sum_{|\beta|=m}2(\p^\beta \p_i f_{m+1},\p^\beta f_m)\right)^2(0) \\ & =4C_m^2\sum_{i=0}^3 \left(\int_{\B_1} (\p_i f_{m+1},f_m)(x) d\lambda_4\right)^2\end{split}
\end{equation}
and
\begin{equation}\label{4}\begin{split}
\Delta u(0)&=2\sum_{i=0}^3 \sum_{|\beta|=m} (\p^\beta \p_i f_{m+1},\p^\beta \p_i f_{m+1})(0)\\&=2 C_m\int_{\B_1} |\p_i f_{m+1}(x)|^2 d\lambda_4 \end{split}
\end{equation}
while 
\begin{equation}
u(0)=\sum_{|\beta|=m} |\p^\beta f_m(0)|^2=C_m\int_{\B_1} |f_m(x)|^2d\lambda_4 .\end{equation}
To prove Proposition \ref{p1} we need to estimate  $ \frac{|\nabla u (0)|^2 }{2u(0)\Delta u(0)}$ which is equal to
\begin{equation}\label{5}\frac{\sum_{i=0}^3 \left(\int_{\B_1} (\p_i f_{m+1},f_m)(x) d\lambda_4\right)^2 }{\left( \int_{\B_1} |f_m(x)|^2d\lambda_4\right)\left(\sum_{i=0}^3 \int_{\B_1} |\p_i f_{m+1}(x)|^2 d\lambda_4 \right)}.
\end{equation}

\begin{definition}
	For $m\in \N$ we define the space of $\H$ valued homogeneous polynomials of degree $m$  
	$$ \mathcal{P}^m(\H,\H):=\{ f(q)=\sum_{|\alpha|=m} c_\alpha \left( x_0^{\alpha_0}x_1^{\alpha_1}x_2^{\alpha_2}x_3^{\alpha_3}\right) =\sum_{|\alpha|=m} c_\alpha x^\alpha,\ c_\alpha\in \H \quad \alpha=(\alpha_0,..,\alpha_3)\in \N^4 \} $$
and the corresponding space of Fueter regular polynomials
$$ \mathcal{M}^m(\H,\H) :=\{ f\in \mathcal{P}^m(\H,\H) | \dbar f = 0\}$$
\end{definition}
Clearly we need to compute  the supremum of \ref{5} for $f_m\in \M^{m}(\H,\H)$ and $f_{m+1}\in \M^{m+1}(\H,\H)$. Note that the vector spaces $\M^m(\H,\H)$ are finite dimensional and we put on them the norm induced by $L^2 (\B_1)$. 
\begin{lemma}\label{lemma2} If $M$ is the following 
\begin{equation} \label{29bis} M:=\sup_{\substack{0\neq f_m\in \M^m \\ 0\neq f_{m+1}\in \M^{m+1}}  }  \left\{ \frac{\sum_{i=0}^3 \left(\int_{\B_1} (\p_i f_{m+1},f_m)(x) d\lambda_4\right)^2 }{\left( \int_{\B_1} |f_m(x)|^2d\lambda_4\right)\left(\sum_{i=0}^3 \int_{\B_1} |\p_i f_{m+1}(x)|^2 d\lambda_4 \right)}\right\} 
\end{equation}
we also have
\begin{equation}
  M= \max_{ 0\neq f_{m+1} \in \M^{m+1}} \left\{ \frac 1{2(m+3)(m+1)} \frac{\|\p_0 f_{m+1}\|^2}{\| f_{m+1}\|^2}  \right\}.
\end{equation}
\end{lemma}
\begin{proof}
Starting from equation \eqref{29bis} it is not restrictive to assume that 
\begin{equation}\label{6}
\begin{split} 
 \int_{\B_1} |f_m(x)|^2d\lambda_4  =1 \text{ and } \\
\sum_{i=0}^3 \int_{\B_1} |\p_i f_{m+1}(x)|^2 d\lambda_4 =1
\end{split}
\end{equation}
and so we have  
\begin{equation}  \label{7} 
M=\sup_{\substack{ f_m \in \M^m, \|f_m\|=1  \\  f_{m+1}\in \M^{m+1}, \sum_{i=0}^{3}\|\p_i f_{m+1}\|^2 =1}  } \left\lbrace \sum_{i=0}^3 \left(\int_{\B_1} (\p_i f_{m+1},f_m)(x) d\lambda_4\right)^2 \right\rbrace . \end{equation}
Since $f_{m+1}$ and $f_m$ are in two finite dimensional spaces we first fix $f_{m+1}$ and calculate the maximum value 
$$m(f_{m+1}):=\sup_{\|f_{m}\|=1} \left\lbrace  \sum_{i=0,..,3} \left( \int_{\B^4} (\p_if_{m+1},f_m) d\lambda_4 \right)^2 \right\rbrace $$
The maximum will be attained on an element $f_m$ such that:
\begin{equation}\label{8} 
\begin{split} \text{ for $h\in \M^m$ with  }
\int_{\B} (f_m,h)d\lambda_4 =0  \text{ we have }  \\ \sum_{i=0}^3\underbrace{\left( \int_{\B} (\p_i f_{m+1},f_m)d\lambda_4\right)}_{\text{call this term $\xi_i$}} \int_\B (\p_i f_{m+1},h)d\lambda_4 =0
\end{split} \end{equation}

it follows that 
\begin{equation}\label{yy}
\sum\xi_i\p_i f_{m+1}=\Lambda f_m
\end{equation}
for some $\Lambda\in \R$. This $\Lambda$ is exactly $m(f_{m+1})$.
By taking the $L^2$-product of \ref{yy} with $f_m$ we have
$$ \int_{\B} (\sum\xi_i\p_i f_{m+1},f_m)d\lambda =\sum_{i=0}^3 \xi_i^2 =\Lambda =m(f_{m+1})$$
and by squaring \refeq{yy}
$$\sum_{i,j=0,..,3} \xi_i\xi_j \int_\B (\p_i f_{m+1},\p_j f_{m+1})d\lambda_4 =m^2(f_{m+1}) .$$ 
Let $X=\frac{1}{\sqrt{m(f_{m+1})}}\sum_i \xi_i\p_i$. Clearly $X$ is a unitary derivative and we have that $$\int_{\B} |Xf_{m+1}|^2d\lambda_4=m(f_{m+1}) .$$
It turns out that $m(f_{m+1})$ is actually the maximum of the following 
\begin{equation}\label{9} \max_{\sum_{i=0}^3 \zeta_i^2=1} \sum_{h,k=0,..,3}\zeta_h\zeta_k (\p_h f_{m+1},\p_k f_{m+1})=\max_{|\zeta|=1}\mathcal{A}(\zeta,\zeta).
\end{equation}
where $\mathcal{A}=(a_{hk})_{h,k=0,..,3}$ is the $4\times 4$ matrix whose entries are
$$a_{hk}=\int_\B (\p_h f_{m+1},\p_k f_{m+1})d\lambda_4 .$$
In fact since $\mathcal{A}$ is symmetric the solution of \ref{9} is the largest eigenvalue of $\mathcal{A}$, say $\Lambda'$, and it is attained at the corresponding eigenvector $\zeta_{\Lambda'}$. Therefore by choosing
$f_m=\frac{1}{\sqrt{\Lambda'}}\sum^3_{i=0} (\zeta_{\Lambda'})_i \p_i f_{m+1} $ we see that it satisfies \ref{8}, hence $\Lambda'=m(f_{m+1})$.
By choosing a suitable $c\in \H$ and considering $\tilde{f}_{m+1}(q):=f_{m+1}(cq)$ we can assume that $X=\p_0$. 
Finally we have
\begin{equation}\label{225} 
M=\max_{\substack{f_{m+1}\in \M^{m+1} \\ \sum_{i=0}^{3} \|\p_i f_{m+1}\|^2 =1}} m(f_{m+1}) =\max_{\substack{f_{m+1}\in \M^{m+1} \\ \sum_{i=0}^{3} \|\p_i f_{m+1}\|^2 =1 }} \{ \|\p_0 f_{m+1}\|^2\} .\end{equation}
We note that the condition $\sum^3_{i=0} \| \p_i f_{m+1}\|^2 =1$ can be expressed in terms of $\|f_{m+1}\|$. If $u$ and $v$ are two homogeneous harmonic polynomials of degree $m+1$ we have:
\begin{equation}\label{10}
\sum^{3}_{i=0}\int_{\B} (\p_i u,\p_i v)d\lambda_4 =
\int_{\mathbb{S}} u\p_\nu v d\Sigma = (m+1)\int_{\mathbb{S}} uv d\Sigma =2(m+1)(m+3)\int_\B uv \lambda_4 .
\end{equation} 
If we apply \ref{10} to the components of $f_{m+1}$ we have that
\begin{equation} \label{10bis}
\sum^3_{i=0} \|\p_i f_{m+1}\|^2 =2(m+1)(m+3)\|f_{m+1}\|^2 .
\end{equation}
Therefore by plugging this into formula \refeq{225} we get our conclusion.
\end{proof} 
\begin{remark}\label{r1}
The maximum problem
$$ \max_{ f_{m+1} \in \M^{m+1}} \left\{ \frac 1{2(m+3)(m+1)} \frac{\|\p_0 f_{m+1}\|^2}{\| f_{m+1}\|^2}  \right\} $$
is a hard one but it simplifies a lot if we can find an orthogonal decomposition  $\M^{m+1}(\H,\H)=\oplus_\mu G_\mu$ with the property that $\p_0 (G_\mu)$ are still orthogonal to each other and such that $\frac{\|\p_0 \tilde{g}_{\mu}\|^2}{\| \tilde{g}_\mu \|^2}$ is constant for $\ \tilde{g}_\mu\in G_\mu$ for fixed  $\mu$. In fact in this way the maximum reduces to
\begin{equation}\label{12}
 \max_{\mu}   \left\{ \frac{1}{2(m+3)(m+1)} \frac{\|\p_0 \tilde{g}_{\mu}\|^2}{\| \tilde{g}_\mu \|^2}  \right\}\end{equation} \end{remark} 
The difficult part is to find such orthogonal decomposition of  $\M^{m+1}(\H,\H)$. In analogy with harmonic functions the spirit of the next step is to express the elements in $\M^{m+1}(\H,\H)$ by using their restriction to the imaginary space ($\simeq \R^3$). This is possible by using the Cauchy-Kowalevski extension operator (which will be denoted by $\tilde{\cdot}$). After that we exploit the Fischer decomposition of polynomials into monogenic functions. To this end we follow \cite{DSS92} and consider the monogenic functions which annihilate the Dirac operator $D$ in $\R^3$ where we shall identify 
$$\R^3 = \{ x \in \H | x_0=0 \} $$
and indicate its elements with $\underline{x}=x_1\i+x_2\j+x_3\k$. 
\begin{definition}
	We define
	$$ \mathcal{M}^m(\R^3,\H)=\left\lbrace f\in \mathcal{P}^m(\R^3,\H) \text{ such that } Df=0\right\rbrace $$
\end{definition}
We have the following
\begin{theorem}[Fischer decomposition] \label{t2}
$$\mathcal{P}^m(\R^3,\H)= \bigoplus_{j=0}^m \underline{x}^j\mathcal{M}^{m-j} (\R^3,\H)  $$
\end{theorem}
Note that this decomposition is orthogonal with respect to the $L^2(\B^3)$ product in the unit ball of $\R^3$ ( see Proposition \ref{p2} below and \cite{DSS92} ). 
We apply this decomposition to the restriction of $f_{m+1}|_{\R^3}$ and get $f_{m+1}(\ux)=\sum_{j=0}^{m+1}\ux^j g_{m+1-j}(\ux)$. 
We recall that given a polynomial $p(\ux)$ on $\R^3$ with values in $\H$ the Fueter regular extension $\tilde{p}$ to $\H$ is given by the Cauchy-Kowalevski operator $\exp (-x_0D) p$ namely
\begin{equation}\label{a}
  \tilde{p} (x)= \sum^{+\infty}_{k=0} \frac{(-x_0 D)^k}{k!}p(\ux) .\end{equation}
Since $f_{m+1}(x)=\widetilde{f_{m+1}|_{\R^3}}( x)$, we have 
\begin{equation}\label{pd1}
f_{m+1}(x)=\sum_{j=0}^{m+1}\widetilde{\ux^j g}_{m+1-j}(x)
\end{equation}
which is the decomposition we are looking for. 
In order to obtain the regular extension of the terms of type $\ux^jg_{m+1-j}(\ux)$ described in the Fischer decomposition we need to compute the $D$ operator on terms of the kind $\ux^s f(\ux)$, $f\in \M^k(\R^3,\H)$.
It is easy to check (see for instance \cite{Be99}) that  
$$ D\ux^s=\begin{cases} -s \ux^{s-1} \text{ if $s$ is even} \\
 -(s+2)\ux^{s-1} \text{ if $s$ is odd} 
\end{cases}
$$ 
therefore if $g_k$ is a homogeneous monogenic polynomial of degree $k$ we have
\begin{equation}\label{b}
 D\ux^s g_k(\ux)=\begin{cases} -s\ux^{s-1}g_k(\ux) \text{ if $s$ is even} \\
-(s+2+2k)\ux^{s-1}g_k(\ux) \text{ if $s$ is odd.} \end{cases}
\end{equation}
Putting equation \refeq{a} and \refeq{b} together
\begin{equation}\label{11}
\widetilde {\ux^s g_k(\ux)}=\left(\sum_{j=0}^s c_{s,k,j}\frac{(x_0)^j}{j!} \ux^{s-j} \right) g_k(\ux). 
\end{equation}
where $c_{s,k,j}$ are real constants. 
We are now able to prove that the decomposition in \eqref{pd1} and its $\partial_0$-dervative are indeed orthogonal decompositions
\begin{proposition} \label{p2}Let $f\in\M^k(\R^3,\H)$ and $g\in\M^h(\R^3,\H)$ be two  homogeneous monogenic polynomials of degree $h$ and $k$. We have 
$$ \int_{\B^3} \bar{g}\ux f d\lambda_3 =0 .$$
Moreover if $h\neq k$ 
\begin{equation}\label{31} \int_{\B^4} \overline{\left( \widetilde{\ux^n g(\ux)} \right)}\left( \widetilde{\ux^m f(\ux)}\right) d\lambda_4 =0\quad \text{ for  }m,n\in\N. \end{equation}
\end{proposition} 
\begin{proof} To prove the first part of the proposition we begin by observing that
\begin{equation}\label{29} \int_{\B^3} \bar{g}\ux f d\lambda_3 =\frac 1{h+k+4}\int_{\mathbb{S}^2} \bar{g}\ux f d\Sigma     .\end{equation}
This last integral, by the divergence formula, is equal to
\begin{equation}\label{30}
 \int_{\B^3} (\overline{(Dg} f +\bar{g} Df )d\lambda_3 =0.\end{equation}
For the second part we start decomposing the integral into slices and applying the first part of the proposition
\begin{align} \label{32}
\int_{\B^4} \overline{\left( \widetilde{\ux^n g(\ux)} \right)} \left( \widetilde{\ux^m f(\ux)}\right) d\lambda_4&= \int_{-1}^1 \left( \int_{\B^3_{\sqrt{1-x_0^2}}}\overline{\left( \widetilde{\ux^n g(\ux)} \right)} \left( \widetilde{\ux^m f(\ux)}\right)       d\lambda_3 
\right)dx_0 \\ \nonumber
 &=\int_{-1}^1 \left( \int_{\B^3_{\sqrt{1-x_0^2}}} \bar{g}(\ux)(p(x_0,\ux))f(\ux)d\lambda_3\right)  dx_0
\end{align}
where $p$ is a polynomial, with real coefficients, in $x_0$ and $\ux$ obtained after replacing the extensions $\tilde{\cdot}$ with their expression as in \ref{11}. 
Note that in the inner integral the terms are of type
$$ c(x_0)\int_{\B^3_{\sqrt{1-x_0^2}}} \bar{g}(\ux)\ux^j f(\ux)d\lambda_3$$
If $j$ is even then $\ux^j=(-1)^{\frac{j}{2}} |\ux|^j$
and integrating on spherical shells the integral is $0$ because $f$ and $g$ have different degree. If $j$ is odd then on spherical shells the integral is $0$ by equations \eqref{29} and \eqref{30}.   
   \end{proof}
\begin{corollary}\label{pc1}Let $m,s_1,s_2$ be positive integers such that $0\le s_1<s_2\le m$. If $f\in \mathcal{M}^{m-s_1}(\R^3,\H)$ and $g\in \mathcal{M}^{m-s_2}(\R^3,\H)$ then $\widetilde{\ux^{s_1}f(\ux)}$ and
$\widetilde{\ux^{s_2}g(\ux)}$ are orthogonal in $L^2(\B)$, moreover $\p_0\widetilde{\ux^{s_1}f(\ux)}$ and $\p_0\widetilde{\ux^{s_2}g(\ux)}$ are orthogonal in $L^2(\B)$.
\end{corollary}
\begin{proof}
The orthogonality of $\widetilde{\ux^{s_1}f(\ux)}$ and
$\widetilde{\ux^{s_2}g(\ux)}$ follows from Proposition \ref{p2} by taking the real part of the integral in equation \eqref{31}, while the orthogonality of $\p_0\widetilde{\ux^{s_1}f(\ux)}$ and
$\p_0\widetilde{\ux^{s_2}g(\ux)}$  follows by noticing that  $$\left( \overline{\p_0\widetilde{\ux^{s_2}g(\ux)}}\right) \p_0\widetilde{\ux^{s_1}f(\ux)}=\bar{g}(\ux)p(x_0,\ux)f(\ux)$$  
where $p$ is a polynomial and therefore the same computations as in \eqref{32} apply.	
\end{proof}
\begin{remark}\label{pr2}
Following the proof of Proposition \ref{p2} it is possible to prove directly with similar computations the orthogonality of $\widetilde{\ux^{s_1} f(\ux)}$  and $\widetilde{\ux^{s_2} g(\ux)}$  i.e. 
$$  \int_{\B^4} \left( \widetilde{\ux^{s_2} g(\ux)}, \widetilde{\ux^{s_1} f(\ux)}\right) d\lambda_4 =0\quad \text{ for  }s_1,s_2\in\N  $$     
\end{remark}
To compute the maximum in \refeq{12} we first give the following
\begin{proposition}\label{p3}
	Let $f\in\M^k(\R^3,\H)$ be a non zero monogenic polynomial. We have
	$$ \frac{\| \p_0\widetilde{\ux^sf(\ux)}\|^2}{\|\widetilde{\ux^sf(\ux)}\|^2 }=\frac{(2k+s+2)(k+s+2)s}{k+s+1} .$$ 
\end{proposition}
\begin{proof}
We first need a better form for \refeq{11}. To this end we follow \cite{DSS92}, we put $r^2=x_0^2 +|\ux|^2$ and note that
the term between brackets in \ref{11} is a homogeneous polynomial of degree $s$ in $x_0$ and $\ux$. The terms with an even power of $\ux$ can be expressed as polynomial functions of $r^2-x_0^2$ and therefore we have that
\begin{equation}\label{14}
\widetilde{\ux^s f(\ux)}=r^s(A(\frac{x_0}r)+B(\frac{x_0}r) \frac \ux r)f(\ux).
\end{equation}
By imposing the condition $\bar\partial \widetilde{\ux^s f(\ux)}=0$ we find that $A$ and $B$ must have the following form:
 
$$ \widetilde{\ux^s f(\ux)}=d_{k,s }r^s\left(C_{s}^{k+1}(\frac {x_0} r) +\frac\ux r \frac{2k+2}{s+2k+2}C^{k+2}_{s-1}(\frac {x_0} r) \right) f(\ux) $$
where $C_{n}^{\mu}$ is the  Gegenbauer's polynomial defined by the relations
$$ \frac{1}{(1-2xt+x^2)^\mu}=\sum_{n=0}^{+\infty}C^\mu_n(t)x^n .$$
and $d_{k,s}$ are some constants of which we will not keep track because they will cancel in the quotient. (for the details of the computations we refer to \cite{DSS92} where basically we have the same calculations done for the Clifford algebra generated by 3 vectors). We have
\begin{equation}\label{15}
\begin{split}
 {}&\frac{\|\widetilde{\ux^2 f(\ux)}\|^2}{d_{k,s}^2}= \int_\B r^{2s}\left( (C_{s}^{k+1}(\frac {x_0} r))^2 +(\frac {|\ux |}r \frac{2k+2}{s+2k+2}C^{k+2}_{s-1}(\frac {x_0} r))^2 \right)|f(\ux)|^2 d\lambda_4  \\
={}& \int_0^1 \left ( \int_{\mathbb{S}^3} \left( (C_{s}^{k+1}( {x_0} ))^2 +(1-x_0^2)( \frac{2k+2}{s+2k+2}C^{k+2}_{s-1}(x_0))^2 \right) |f(\ux)|^2d\Sigma \right) r^{2s+2k +3} dr \\
={}& \frac{1}{2(s+k+2)}\int_{-1}^{1} \left( 
\int_{{\sqrt{1-x_0^2}\mathbb{S}^2}}
\left( (C_{s}^{k+1}( {x_0} ))^2 {}+(1-x_0^2)( \frac{2k+2}{s+2k+2}C^{k+2}_{s-1}(x_0))^2 \right) \right. \\ 
&{} 
\times|f(\ux)|^2\frac{d\Sigma_2}{\sqrt{1-x_0^2}} 
\Bigg)dx_0  \\ 
={}&
\frac{1}{2(s+k+2)}\int^1_{-1} (1-x_0^2)^{k+\frac 12}
\left( (C_{s}^{k+1}( {x_0} ))^2 +(1-x_0^2)( \frac{2k+2}{s+2k+2}C^{k+2}_{s-1}(x_0))^2 \right)dx_0 \\ 
&{} \times\int_{\mathbb{S}^2}|f(\ux)|^2 d\Sigma_2 .
\end{split}
\end{equation}
We recall that
\begin{equation}\label{21} \int ^1_{-1} (C^{\mu}_\nu(x))^2 (1-x^2)^{\mu-\frac 12} dx=\frac{2^{1-2\mu} \Gamma(\nu+2\mu) \pi}{\nu! (\nu+\mu)\Gamma(\mu)^2}\end{equation}
and the last line of \ref{15} becomes
\begin{equation}\label{16}
\frac{1}{2(s+k+2)}
\frac{2^{-2k}\pi \Gamma(s+2k+2)}{(s)!\Gamma(k+1)^2}\left( \frac{1}{(2k+s+2)} 
\right)\int_{\mathbb{S}^2}|f(\ux)|^2 d\Sigma_2 .
\end{equation}
Since $\p_0 (r)=\frac{x_0}r$ and $\p_0({\frac{x_0}r})=\frac{r^2-x^2_0}{r^3}$ we have that 
\begin{equation}\label{16bis}
\begin{split}
\frac{\p_0 \widetilde{\ux^s f(\ux)}}{d_{k,s}}={}& r^{s-1}\left(s\frac{x_0}{r}C_{s}^{k+1}(\frac {x_0} r)  +\left( C_{s}^{k+1 '}(\frac {x_0} r)\right) (1-(\frac{x_0}r)^2)  \right. \\  &+ \frac\ux r \left. \frac{2k+2}{s+2k+2}\left( (s-1)\frac{x_0} r C^{k+2}_{s-1}(\frac {x_0} r) +C^{k+2 '}_{s-1}(\frac {x_0} r)(1-(\frac{x_0}r)^2)  \right) \right) f(\ux).
\end{split}
\end{equation}
Since the Gegenbauer polynomials satisfy the equation
\begin{equation}\label{xx} (1-t^2)C^{\mu '}_s(t)+stC^{\mu}_s(t)=(s+2\mu-1)C^\mu_{s-1}(t)
\end{equation}
by plugging into \ref{16bis} we have
\begin{equation}
\frac{\p_0 \widetilde{\ux^s f(\ux)}}{d_{k,s}}= r^{s-1}(s+2k+1)\left( (C_{s-1}^{k+1}(\frac {x_0} r) +
 \frac\ux r  \frac{(2k+2)}{s+2k+1}C^{k+2 }_{s-2}(\frac {x_0} r)\right)f(\ux) .
\end{equation}
Now we repeat the computations that we have done in \ref{15} and get
$$ \frac{\|\p_0 \widetilde{\ux^s f(\ux)}\|^2}{d^2_{k,s}} =
\frac{(s+2k+1)}{2(s+k+1)}
\frac{2^{-2k}\pi \Gamma(s+2k+1)}{(s-1)!\Gamma(k+1)^2}\int_{\mathbb{S}^2}|f(\ux)|^2 d\Sigma_2 .
$$
Finally we have that
$$  \frac{ \|\p_0 \widetilde{\ux^s f(\ux)}\|^2}{ \| \widetilde{\ux^s f(\ux)}\|^2} =\frac{(2k+s+2)(k+s+2)s}{k+s+1}.$$
\end{proof} 
\begin{proof}[Proof of Proposition \ref{p1}]
 We have
\begin{equation}\label{pf1} 
\begin{split}
\frac{|\nabla u|^2}{2u\Delta u}\overset{\textrm{Lemma \ref{lemma2}}}{\leq}&  \max_{ f_{m+1} \in \M^{m+1}} \left\{ \frac 1{2(m+3)(m+1)} \frac{\|\p_0 f_{m+1}\|^2}{\| f_{m+1}\|^2}  \right\}\\
  \overset{\textrm{Remark \ref{r1}}}{=}& \max_{\substack{0\leq\mu\leq m+1 \\ g_{m+1-\mu}\in \M^{m+1-\mu}(\R^3,\H)}}   \left\{ \frac 1{2(m+3)(m+1)} \frac{\|\p_0 \widetilde{\ux^\mu g}_{m+1-\mu}(x)\|^2}{\| \widetilde{\ux^\mu g}_{m+1-\mu}(x) \|^2}  \right\} 
  \end{split}
  \end{equation}
Applying Proposition \ref{p3} to the ratio $\frac{\|\p_0 \widetilde{\ux^\mu g}_{m+1-\mu}(x)\|^2}{\| \widetilde{\ux^\mu g}_{m+1-\mu}(x) \|^2}$, where $k=m+1-\mu$ and $s=\mu$, we have that: the last maximum in \eqref{pf1} is attained for $\mu=m+1$ and its value is $\frac{m+3}{2(m+2)}$. \end{proof}
\begin{proof}[Proof of Theorem \ref{t1}]
The conclusion follows immediately by applying Lemma \ref{l1} and Proposition \ref{p1}.  	
\end{proof}
\section{Subharmonicity of higher gradients of monogenic functions on Clifford algebras}\label{s3}
In this section we will repeat the same computations for monogenic functions on Clifford algebras. Let $\R^n$ be the standard vector space of dimension $n$ and  $\e_1,...,\e_n$ be the canonical base. We shall indicate by  $\R_n$ the Clifford algebra generated by these vectors. Every element $\mathbf{a}\in\R_n$ can be written as
$$ \mathbf{a}=\sum_{\mathclap{A\subset \left\lbrace 1,...,n\right\rbrace }} a_A\e_A$$
where $a_A\in\R$ and $\e_A=\e_{i_1}\cdots\e_{i_k}$ with $A=\left\lbrace i_1<...<i_k\right\rbrace$  and by convention $\e_\emptyset =1$. We call $a_\emptyset$ the real part of $\mathbf{a}$.
We recall that $\R_n$ is endowed with a non degenerate scalar product with respect to which the powers $\e_A$ form an orthonormal system. This scalar product is defined in terms of an involution which generalizes the conjugation in $\R_n$ namely we define on the elements of the canonical base:
$$ \overline{\e_A}:=(-1)^{k} \e_A^*$$
where $\e_A^*=\e_{i_k}\cdots \e_{i_1}$  and extend this definition by linearity to the full $\R_n$. The scalar product between two elements $\mathbf{a}$ and $\mathbf{b}$ is
\begin{equation}
(\mathbf{a},\mathbf{b})=(\overline{\mathbf{b}}\mathbf{a})_\emptyset
\end{equation}
and moreover $|\mathbf{a}|^2=(\mathbf{a},\mathbf{a})=\sum_{A} a^2_A$.
We identify $\R^n$ inside $\R_n$ by $x=(x_1,...,x_n)=x_1\e_1+...+x_n\e_n$.
\begin{definition} Let $\Omega$ be an open subset of $\R^n$ and let $f:\Omega \rightarrow \R_n$ be a $C^1$ function. We say that $f$ is monogenic if
	$$ D_n f(x)=(\e_1\p_1+...+\e_n\p_n )f=0$$
	\end{definition}
\begin{definition}
	We define $\P^k(\R^n,\R_n)$ as the space of homogeneous polynomials of degree $k$ with values in $\R_n$. Similarly we define 
	$$ 
	\M^k(\R^n,\R_n)=\left\lbrace f\in \P^k(\R^n,\R_n) | D_nf=0\right\rbrace  $$    
\end{definition}
 Let $\beta=(i_1,...,i_m)$ be a multi-index, where $i_j\in\{1,...,n\} \forall j$, and define $\p^\beta f(x)=\p_{i_1}\cdots \p_{i_m}f(x)$.
\begin{definition} Let $f$ as before and $m$ a positive integer. We define the $m$-th gradient $\nabla^mf$ as the set of all derivatives $\p^\beta f(x)$ for all $\beta$ with lenght $m$ and set $$ u=|\nabla^m f|^2=\sum_{|\beta|=m} |\p^\beta f|^2$$ 
\end{definition}
We want to find precisely for which $\alpha$ we have that $u^{\frac \alpha 2}$ is subharmonic.
We follow the same proof of the preceding section. In particular if $f_m, f_{m+1}$ are monogenic homogeneous polynomials of degree $m$ and $m+1$ we look for the maximum $M$ :
\begin{equation}\label{17}M:=\max_{\substack{f_m\in\M^m(\R^n,\R_n) \\ f_{m+1}\in\M^{m+1}(\R^n,\R_n)} } \frac{\sum_{i=1}^n \left(\int_{\B_1} (\p_i f_{m+1},f_m)(x) d\lambda_4\right)^2 }{\left( \int_{\B_1} |f_m(x)|^2d\lambda_4\right)\left(\sum_{i=1}^n \int_{\B_1} |\p_i f_{m+1}(x)|^2 d\lambda_4 \right)}.
\end{equation}
It is similar, as in the previous paragraph, to see that such maximum is equivalent to
\begin{equation}\label{18} M=\max_{\substack{ \sum_{i=1}^n \zeta_i^2=1 \\ f_{m+1}\in \mathcal{M}(\R^n,\R_n)\\ \sum_{i=1}^n \|\p_i f_{m+1}\|^2=1}} \left\{\sum_{h,k=1,..,n}\zeta_h\zeta_k (\p_h f_{m+1},\p_k f_{m+1})\right\}.
\end{equation}
  For every fixed $f_{m+1}$ in \eqref{18} the maximum is reached for some $\zeta \in \R^n$. Up to the choice of an $s\in Spin (n)$ we can assume, by considering $sf_{m+1}(s^{-1}xs)$ that $\zeta =(0,...,1)$ or, in other words, that $\p_\zeta =\p_n$ (see \cite{DSS92} section 1.12.2).
  As was shown in \eqref{10bis} we have that
  $$ \sum_{i=1}^n \|\p_i f\|^2=(m+1)(2m+n+2)\|f\|^2 $$
  and 
  \begin{equation}\label{24}  M= \max_{ f_{m+1} \in \mathcal{M}^{m+1}(\R^n,\R_n)} \left\{ \frac 1{(2m+n+2)(m+1)} \frac{\|\p_1 f_{m+1}\|^2}{\| f_{m+1}\|^2}  \right\} .
  \end{equation}
In order to compute $M$ we need to find an orthogonal decomposition of $\mathcal{M}^{m+1}(\R^n,\R_n)$ similar to that in Remark \ref{r1} (see \eqref{pd2}). To this end we exploit again the Fischer decomposition theorem. We consider the splitting $\R^n= \R^{n-1}\oplus \R\e_n$ and identify in this manner $\R^{n-1}$ inside $\R^n$. With this identification we take $\M^k(\R^{n-1},\R_n)$ to be the space of monogenic homogeneous polynomials in $x_1,....,x_{n-1}$ of degree $k$ with values in $\R_n$. Moreover we shall indicate with $x=x_1\e_1+...+x_n\e_n$ a vector in $\R^n$ and with $\ux =x_1\e_1+...+x_{n-1}\e_{n-1}$ a vector in $\R^{n-1}$. 
\begin{theorem}[Fischer decomposition in $\R^n$, \cite{DSS92} ] We have the following decomposition 
$$ \mathcal{P}^k(\R^{n-1}, \R_n)=\bigoplus_{j=0}^k \ux^j \M^{k-j}(\R^{n-1},\R_n) $$
\end{theorem} 
We note that this decomposition is orthogonal also in $L^2(\B_{n-1})$.
As we did in the previous section starting from $f_{m+1}\in\mathcal M^{m+1}(\R^n,\R_n)$ and using the Fischer decomposition to its restriction $f_{m+1}|_{\R^{n-1}}$, we can first decompose $f_{m+1}(\ux)=\sum_{s=0}^{m+1}\ux^s g_{m+1-s}(\ux)$ where $g_{m+1-s}\in\mathcal M^{m+1-s}(\R^{n-1},\R_n)$ and then we can regain $f_{m+1}(x)$ considering its monogenic extension using the Cauchy-Kowalevski operator $\exp(x_n \e_n D_{n-1})$: 
\begin{equation}\label{pd2}
f_{m+1}(x)=\sum_{s=0}^{m+1}\widetilde {\ux^s g_{m+1-s}}(x)
\end{equation}
where $\widetilde {\ux^s g_{m+1-s}}(x)=\exp(x_n \e_n D_{n-1}) (\ux^sg(\ux))$.
\begin{proof}[Proof of Theorem \ref{t3}]
The proof goes as the one of Theorem \ref{t1}, we only need to adapt Proposition \ref{p3}. 
We consider elements of the type $\ux^s f(\ux)$ where $f$ is a monogenic homogeneous polynomial of degree $k$ in $\R^{n-1}$ (as in the decomposition \eqref{pd2}). It follows that
\begin{equation} \label{19} 
\begin{split}
\widetilde{\ux^s f(\ux)}=& \sum_{i=0}^s \left( c_{s,i,n}x^i_n \ux^{s-i}\e_n^i\right) f(\ux) \\
=& |x|^s \left( A(\frac{x_n}{|x|})+B(\frac{x_n}{|x|} )\frac {(\ux)}{|x|} \e_n \right) f(\ux)
\end{split}
\end{equation}
where $c_{s,i,n}$ are some convenient real constants. We look for an explicit formula for $A$ and $B$ and since $D_n\widetilde{\ux^s f(\ux)} =0$, following \cite{DSS92}, we have that 
$$ \widetilde{\ux^s f(\ux)}=d_{n,k,s}|x|^s\left( C^{\frac n2 +k-1}_{s} (\frac {x_n}{|x|}) +\frac{n +2k-2}{n+2k+s-2} C^{\frac{n}2 +k}_{s-1}(\frac {x_n}{|x|}) \frac{\ux}{|x|}\e_n \right) f(\ux)$$
 where $d_{n,k,s}$ is a constant which depends only on $n,k$ and $s$ (see also \cite{L12}).
We compute next $\|  \widetilde{\ux^s f(\ux)}\|^2$ and we have

\begin{equation}\label{20}
\begin{split}
{}&\frac{\|\widetilde{\ux^s f(\ux)}\|^2}{d_{n,k,s}^2}= \int_{\B_n} r^{2s}\left( (C_{s}^{\frac n2+ k-1}(\frac {x_n} {|x|}))^2 +(\frac {|\ux |}{|x|} \frac{n+2k-2}{n+2k+s-2}C^{\frac n2 +k}_{s-1}(\frac {x_n} {|x|}))^2 \right)\\&\times |f(\ux)|^2 d\lambda_n   \\
={}& \int_0^1 \left ( \int_{\mathbb{S}^{n-1}} \left( (C_{s}^{\frac n2 +k-1}( {x_n} ))^2 +(1-x_n^2)( \frac{n+2k-2}{n+2k+s-2}C^{\frac n2 +k}_{s-1}(x_n))^2 \right)|f(\ux)|^2d\Sigma \right) \\ &{}\times r^{n+2s+2k-1} dr \\
={}& \frac{1}{n+2s+2k}\int_{-1}^{1} \left( 
\int_{{\sqrt{1-x_0^2}\mathbb{S}^{n-2}}}
  \left( (C_{s}^{\frac n2 +k-1}( {x_n} ))^2 {}+(1-x_n^2)( \frac{n+2k-2}{n+2k+s-2} C^{\frac n2 +k}_{s-1}(x_n))^2 \right) \right.\\{}&\left.  \times |f(\ux)|^2\frac{d\Sigma_{n-2}}{\sqrt{1-x_n^2}} \right)dx_n
\\ ={}&
\frac{1}{n+2s+2k}\int^1_{-1} (1-x_n^2)^{\frac n2 +k -\frac 32}
\left( (C_{s}^{\frac n2 +k-1}( {x_n} ))^2 +(1-x_n^2)( \frac{n+2k-2}{n+2k+s-2}C^{\frac n2 +k}_{s-1}(x_n))^2 \right)dx_n \\  &{} \times\int_{\mathbb{S}^{n-2}}|f(\ux)|^2 d\Sigma_2 .
\end{split}
\end{equation}
by \refeq{21} we have 
\begin{equation}\label{22}
\frac{\|\widetilde{\ux^s f(\ux)}\|^2}{d_{n,k,s}^2}=\frac{ 2^{4-n-2k} \Gamma(n+2k+s-2)\pi}{s!(n+2k+s-2)(n+2k+2s)\Gamma(\frac n2 +k-1)^2 }\int_{\mathbb{S}^{n-2}}|f(\ux)|^2 d\Sigma_{n-2}
\end{equation}
Similarly for $ \partial_n \widetilde{\ux^s f(\ux)}$:
\begin{equation}
\begin{aligned} \frac{\partial_n \widetilde{\ux^s f(\ux)}}{d_{n,k,s}}&=|x|^{s-1}\left( sC^{\frac{n}{2} +k-1}_{s}(\frac{x_n}{|x|})\frac{x_n}{|x|} +C^{\prime\frac{n}{2} +k-1}_{s}(\frac{x_n}{|x|}) (1-(\frac{x_n}{|x|})^2) \right. \\ &\phantom{=} + { } 
(\frac{n+2k-2}{n+2k+s-2})((s-1)C^{\frac{n}{2} +k}_{s-1}(\frac{x_n}{|x|})\\ &\phantom{=} { }+ \left.
C^{\frac{n}{2} +k \prime}_{s-1}(\frac{x_n}{|x|})(1-(\frac{x_n}{|x|})^2) 
) \frac{\ux}{|x|} \e_n \right) f(\ux) \end{aligned}
\end{equation}
and by equation \refeq{xx} we have
\begin{equation}
\frac{\partial_n \widetilde{\ux^s f(\ux)}}{d_{n,k,s}}= (n+2k+s-3)|x|^{s-1}\left( 
C^{\frac{n}{2} +k-1}_{s-1}(\frac{x_n}{|x|}) +\frac{n+2k-2}{n+2k+s-3}C^{\frac{n}{2} +k}_{s-2}(\frac{x_n}{|x|}) \frac{\ux}{|x|} \e_n \right)f(\ux).
\end{equation} 
It follows by applying the same computation as in \refeq{22} with $s-1$ in place of $s$ that
\begin{equation}\label{23}
\begin{split}
\frac{\|\partial_n \widetilde{\ux^s f(\ux)}\|^2}{d_{n,k,s}^2}&=
\left(\frac{ 2^{4-n-2k} \Gamma(n+2k+s-3) (n+2k+s-3)\pi}{(s-1)!(n+2k+2s-2)\Gamma(\frac n2 +k-1)^2 }\right) \times \\ 
&\phantom{=} {} \times \int_{\mathbb{S}^{n-2}}|f(\ux)|^2 d\Sigma_{n-2}.
\end{split}
\end{equation}
 Taking the quotient between \refeq{22} and \refeq{23} we have
 \begin{equation}
 \frac{\|\partial_n \widetilde{\ux^s f(\ux)}\|^2 }{\| \widetilde{\ux^s f(\ux)}\|^2}= \frac{s(n+2s+2k)(n+2k+s-2)}{n+2k+2s-2}.
 \end{equation}
 Since $k+s=m+1$, the maximum $M$ is reached for $s=m+1$ and by plugging into \refeq{24} we have 
 \begin{equation}
 M= \frac{n+m-1}{n+2m}. 
 \end{equation}  
 Finally we have that $|\nabla^m f|^\alpha$ is subharmonic for $\alpha\ge \frac {2M-1}M =\frac {n-2}{n+m-1}$.
\end{proof}
\section{Monogenic functions on the octonions}
We consider the case of the octonions $\O$.
This algebra is built by the well known Cayley-Dickson construction. Starting from $\H$ we consider on $\H^2$ the following binary operation
$ \cdot :\H^2\times \H^2 \rightarrow \H^2$
\begin{equation}\label{25}
(a,b)\cdot (c,d) =(ac-d\bar{b},cb+\bar{a}d )
\end{equation}
We define $\O$ to be $\H^2$ equipped with the usual sum and with the product $\cdot$ defined in \eqref{25}.
It turns out that $\O$ with the norm inherited from $\H^2$  is a composition algebra which means
$$ | (a,b)\cdot (c,d)|^2 =| (a,b)|^2 |(c,d)|^2 .$$
The dimension over $\R$ of $\O$ is $8$ and a basis is given by
$\e_0=(1,0),\e_1=(\i,0),...,\e_7=(0,\k)$. We note that $\e_j^2 =-1$ for $j>0$, $\e_0^2=1$ and $\e_i\e_j=-\e_j\e_i$ for $i\neq j$, $i>0$ and $j>0$. Clearly $\e_0$ is the unity of $\O$ and for this reason it is also denoted by $1$.
We have that $\O$ splits naturally in two subspaces called the real and the imaginary part (here identified with $\R^7$)
$$ \O \simeq \R \e_0\oplus \underbrace{\R\e_1\oplus...\oplus\R\e_7}_{\cong\R^7} .$$
For every $x\in\O$ we write
$$ x= x_0\e_0 +\sum_{i=1}^{7}x_i\e_i=x_0 +(\ux) $$
$x_0$ is called the real part of $x$ and $\ux$ the imaginary part of $x$. 
We define the conjugate of an octonion $x$  as $\bar{x}=x_0-\ux$. It is easy to check that $\overline{xy}=\bar{y}\bar{x}$ and that the bilinear map $\langle x,y\rangle := (\bar{y}x)_0$ defines a real scalar product on $\O$ and such that $\langle x,x\rangle =|x|^2$. 
We define the Weyl operator
$$\bar\partial =\sum_{i=0}^{7} \e_i\partial_i $$
and the Dirac operator
$$D=\p_{\ux}=\sum_{i=1}^{7}\e_i\partial_i .$$  
\begin{remark}\label{27}
The product defined in \ref{25} is not associative but the following hold:
\begin{align}\label{26}
&x(ax)=(xa)x, \ (xx)a=x(xa),\forall x,a\in \O \\
&\bar{x}(xa)=|x|^2 a
\end{align}

in particular it makes sense to consider the powers in $\O$. Moreover the subalgebra generated by two octonions is associative. (see \cite{CS06} page 76) 
\end{remark}
We have 
$$\partial\dib =\dib\di =\Delta $$ 
and $$ D^2= -\Delta_{x_1,...,x_7} .$$
\begin{definition}
A $C^1$ function $f:\Omega\rightarrow \O$, where $\Omega$ is an open subset of $\O$ (or $\R^7$) is monogenic if
$$ \dib f=0 \ (\text{ resp. }Df=0) $$ 
\end{definition}
We introduce for a positive integer $k$ the space of homogeneous polynomials of degree $k$ on $\O$ (or $\R^7$) $\P^k(\O,\O)$ (resp. $\P^k(\R^7,\O)$) and the corresponding space of monogenic polynomials $\M^k(\O,\O)$ (resp. $\M(\R^7,\O)$). Similarly for a multi-index $\beta\in \left\lbrace 0,...,7\right\rbrace ^k$ we introduce $\p^\beta f=\p_{\beta_1}...\p_{\beta_k} f$ and define $|\nabla^k f(x) |^2 :=\sum_{|\beta |=k} |\p^\beta f(x)|^2 $. In order to prove that $|\nabla^m f|^\alpha$ is subharmonic we exploit again Lemma \ref{l1} and the following
\begin{proposition}\label{p4}
	Let $f:\Omega\rightarrow \O$ be a monogenic function and let $u(x)=\sum_{|\beta|=m} |\p^\beta f(x)|^2$. Then we have
\begin{equation}
\frac{|\nabla u|^2}{2 u\Delta u} \le \frac{m+7}{2(m+4)}
\end{equation}
\end{proposition}
Before getting to the proof of Proposition \ref{p4} we need some preliminaries:
\begin{theorem}
	The following decompositions hold:
	$$ \P^k(\O,\O)=\bigoplus^{k}_{j=0} \bar{x}^j\M^{k-j}(\O,\O) $$
	and similarly
	$$ \P^k(\R^7,\O)=\bigoplus_{j=0}^{k} \ux^j\M^{k-j}(\R^7,\O)$$
\end{theorem}
\begin{proof}
	We introduce on $\P^k (\O,\O)$ a scalar product. For $R_i(x)=\sum_{|\alpha| =k} a^i_\alpha x^\alpha$ for $i=1,2$ (here $\alpha=(\alpha_0,...,\alpha_7)\in \N^8$ is a multiindex, $ |\alpha|=\alpha_0+...+\alpha_7$ and $ x^\alpha =x_0^{\alpha_0}...x_7^{\alpha_7}$) we define
		\begin{align}
	 \left\langle R_1(x),R_2(x)\right\rangle &:=
	 \left( \left( \sum_{|\alpha |=k} \bar{a}_\alpha^1 \p^\alpha \right) \left( \sum_{|\alpha| =k} a^2_\alpha x^\alpha\right) 
	\right)_0 \\& = \left( \sum_{|\alpha|=k}   \bar{a}^1_\alpha a^2_\alpha \alpha!\right)_0  
	 \end{align}
Since $\left(  a(bc)\right)_0 =\left(  (ab)c\right)_0 $ for all $a,b,c\in \O$ we have that for $R_1\in \P^{k-1}(\O,\O)$
\begin{equation*}
\left\langle \bar{x} R_1(x),R_2(x)\right\rangle =\left\langle R_1(x),\dib R_2(x)\right\rangle 
\end{equation*} 
from which follows $\M^k(\O,\O) \subset \left\lbrace \bar{x}\P^{k-1}(\O,\O)\right\rbrace ^\perp$. For the opposite inclusion suppose that $\dib R_2 \neq 0$ and choose $R_1=\dib R_2$. Then clearly $\left\langle \bar{x}R_1,R_2\right\rangle =\left\langle  R_1,\dib R_2\right\rangle \neq 0$ hence $R_2 \notin \left\lbrace \bar{x}\P^{k-1}(\O,\O)\right\rbrace ^\perp$. We have the splitting 
\begin{equation}
\P^k =\M^k\oplus \bar{x}\P^{k-1}
\end{equation} 
and by repeating the same argument on $\P^{k-1}$ we have the conclusion.
The proof of the second decomposition is similar, with $\dib$ replaced by $D$  
\end{proof}
As we did in the first and second section we need to find an orthogonal decomposition of $\M^{m+1}(\O,\O)$ similar to that described in Remark \ref{r1}. Starting from a polynomial $f\in\P^k(\R^7,\O)$ the monogenic extension to $\O$ is given by the Cauchy-Kowalevski extension operator :
$$ \tilde{f}(x)=\sum_{i=0}^{\infty} (-1)^i \frac{x_0^i}{i!} D^i f(\ux) .$$
 Combining the Fischer decomposition and the Cauchy-Kowalevski extension we can decompose any $f_{m+1}\in\mathcal M^{m+1}(\O,\O)$ in the same way as we did in the previous sections
 $$ f_{m+1}(x)=\sum_{s=0}^{m+1}\widetilde{\ux^s g_{m+1-s}}(x). $$
We observe that the orthogonality of this decomposition and of its $\partial_{x_0}$-derivative can be proved as we did in the Proposition \ref{p2} and Corollary \ref{pc1} with foresight to use the real scalar product instead of the hermitian one (see Remark \ref{pr2}). This is because we need the associativity when we compute explicitly 
$$\int_{\mathbb B^8}(\widetilde{\ux^ng(\ux)},\widetilde{\ux^mf(\ux)})_0\, dV$$.  

We need to compute $D (\ux^s f(\ux))$ with $f\in\M^k(\R^7,\O)$. For this we need the following lemmas
\begin{lemma}\label{28}
	Let $f\in\M^k(\R^7,\O)$ then
	$$ D (\ux f(\ux)) =-(2k+7)f(\ux).$$ 
\end{lemma}
\begin{proof}
	We have 
	$$ D (\ux f(\ux))=\sum_{i=1}^{7} \e_i \p_i (\ux f)=-7 f(\ux)+\sum_{i=1}^{7} \e_i(\ux \p_if(\ux)) $$
	For evaluating the last term on the right side we observe that for $t\in \O$ we have
	$$ \left\langle \e_i(\ux\p_if(\ux)),t\right\rangle =\left\langle \ux\p_if(\ux),-\e_i t\right\rangle =-2\left\langle \ux,\e_i\right\rangle \left\langle \p_i f(\ux),t\right\rangle +\left\langle \ux t,\e_i\p_if(\ux)\right\rangle $$
	by the braid and exchange properties (see \cite{CS06}). 
	Taking the sum for $i=1,...,7$ we have
	$$ \left\langle \sum_{i=1}^{7} \e_i(\ux \p_if(\ux)),t\right\rangle =\left\langle -2\underbrace{\sum_{i=1}^{7}x_i\p_i f}_{=kf(\ux) },t\right\rangle +\left\langle \ux t,\underbrace{ D f(\ux)}_{=0}\right\rangle $$ 
	and since it holds for all $t$ we have 
	$$ \sum_{i=1}^{7} \e_i(\ux \p_if(\ux))= -2kf(\ux)$$
	and this yields the conclusion. 
\end{proof}
\begin{lemma}\label{l4}
	Let $f\in \M^k(\R^7,\O)$ and $s$ a positive integer. The following holds
	\begin{equation}
	D (\ux^s f(\ux))=\begin{cases}
	-s \ux^{s-1} f(\ux) \text{ if $s$ is even } \\
	-(s+6+2k)\ux^{s-1}f(\ux)\text{ if $s$ is odd}
	\end{cases}
	\end{equation}
\end{lemma}
\begin{proof}
	We begin with the case when $s=2n+1$. Thanks to Remark \ref{27} we have that
	
	\begin{align*}
    D(\ux^{2n+1} f(\ux))=&(-1)^n D(|\ux|^{2n} \ux f(\ux))\\ 
    =&(-1)^n n|\ux|^{2n-2} \ux^2f(\ux) +(-1)^n |\ux|^{2n} \sum_{i=1}^{7} \e_i(\e_i f(\ux)) \\
    &+(-1)^n|x|^{2n}\sum_{i=0}^{7} \e_i(\ux \p_if(\ux))\\
    =&-(s+6+2k)\ux^{s-1}f(\ux)
    \end{align*}
    
	where the last line follows from Lemma \ref{28}. The case $s$ even is similar.  
	
\end{proof}
 We can now proceed as in the other cases
\begin{proof}[Proof of Proposition \ref{p4}]
	By a similar argument we see that at a point, say $0$, we have that
	$$\frac{|\nabla u|^2}{2u\Delta u}(0)=\frac{\sum_{i=0}^{7} \left( \int_{\B^8}\left[ \p_i f_{m+1},f_m\right] (x)d\lambda_8\right) ^2 }{\left( \int_{\B^8} |f_m(x)|^2d\lambda_8\right) \left( \sum_{i=0}^{7}\int_{\B^8} |\p_i f_{m+1}(x)|^2 d\lambda_8\right)}$$
	where $f_m$ and $f_{m+1}$ are the terms of degree $m$ and $m+1$  in the Taylor expansion of $f$ at $0$. 
	
	The proof reduces to prove that $M=\frac{m+7}{2(m+6)}$ where
	\begin{equation*}
	M :=\max_{\substack{0\neq f_m\in \M^m(\O,\O)\\0\neq f_{m+1}\in \M^{m+1}(\O,\O) } } \frac{\sum_{i=0}^{7} \left( \int_{\B^8}\left[ \p_i f_{m+1},f_m\right] (x)d\lambda_8\right) ^2 }{\left( \int_{\B^8} |f_m(x)|^2d\lambda_8\right) \left( \sum_{i=0}^{7}\int_{\B^8} |\p_i f_{m+1}(x)|^2 d\lambda_8\right)}. 
	\end{equation*}
	By repeating the same argument as in section \ref{s2}, and using the fact that $f(ux)$ is monogenic for all $u\in \O$, we have
	\begin{equation}
	M=\max_{ f_{m+1} \in \M^{m+1}(\O,\O)} \left\lbrace \frac{1}{2(m+5)(m+1)}\frac{\|\p_0 f_{m+1}\|^2}{\| f_{m+1}\|^2}\right\rbrace. 
	\end{equation} 
	At this point we only need to find an 
	orthogonal decomposition of $\M^{m+1} (\O,\O)$ of the type described in Remark \ref{r1}. 
	We start by finding the monogenic extension to $\O$ of $\ux^j f(\ux)$ where $f\in \M^k(\R^7,\O)$. By Lemma \ref{l4} we have
	$$ \widetilde{\ux^jf(\ux)}=\sum_{i=0}^{j} (c_{k,j,i} x_0^i\ux^{j-i})f(\ux) $$
	for some real coefficients $c_{k,j,i}$ which do not depend on $f$. Let $r^2=|x|^2$ and set $\widetilde{\ux^jf(\ux)}=r^j(A(\frac{x_0}{r})+B(\frac{x_0}{r})\frac{\ux}{r})f(\ux)$ and by
	imposing $\dib \widetilde{\ux^jf(\ux)}=0$ we found that the extension is given by
	$$ \widetilde{\ux^j f(\ux)}=d_{k,j}|x|^j\left( C^{ k+3}_{j} (\frac {x_0}{|x|}) +\frac{ 2k+6}{2k+j+6} C^{4 +k}_{j-1}(\frac {x_0}{|x|}) \frac{\ux}{|x|} \right) f(\ux).$$
	The computations are exactly like the ones in section \ref{s3} with $n=8$. In the end we have $M=\frac{m+7}{2(m+4)}$ which finishes the proof. 
\end{proof}
\begin{proof}[Proof of Theorem \ref{t4}]
	Follows from Proposition \ref{p4} and Lemma \ref{l1}.
\end{proof}

\begin{remark}
In Theorems \ref{t1}, \ref{t3} and \ref{t4} we saw that, according to the dimension $n$ of the algebra where $f$ takes its values, $|\nabla^m f|^{\alpha_{0,m,n}}$ is subharmonic for $\alpha_{0,m,n}=\frac{n-2}{n+m-1}$. As observed in \cite{CZ64} Theorem $2$, this is the best possibile choice for the exponent $\alpha_{0,m,n}$ indeed if $\phi(|\nabla^m f|)$ is subharmonic for any monogenic or regular functions $f$ and $\phi$ is continuous then $\phi(t)=\omega(t^{\alpha_{0,m,n}})$ where $\omega:\R_{\geq 0}\to \R$ is a convex increasing function. The proof of this fact follows without substantial modification the proof in \cite{CZ64}.     
\end{remark}
\section*{Acknowledgements}
 The authors wish to thank Professor Irene Sabadini for useful discussions. 
 
 								\bibliographystyle{alpha}

 							\end{document}